\documentclass[12pt]{article}
\usepackage{amssymb, amsmath, amsthm, amscd}
\usepackage[utf8]{inputenc}
\usepackage[all,cmtip]{xy}
\usepackage{enumitem}
\usepackage{setspace}
\usepackage{mathdots}

\usepackage[mathscr]{euscript}
\usepackage[colorlinks=true]{hyperref}
\hypersetup{colorlinks   = true}
\hypersetup{linkcolor=blue}
\usepackage{tikz}

\begin{document}

\mathchardef\mhyphen="2D
\newtheorem{The}{Theorem}[section]
\newtheorem{Lem}[The]{Lemma}
\newtheorem{Prop}[The]{Proposition}
\newtheorem{Cor}[The]{Corollary}
\newtheorem{Rem}[The]{Remark}
\newtheorem{Obs}[The]{Observation}
\newtheorem{SConj}[The]{Standard Conjecture}
\newtheorem{Titre}[The]{\!\!\!\! }
\newtheorem{Conj}[The]{Conjecture}
\newtheorem{Question}[The]{Question}
\newtheorem{Prob}[The]{Problem}
\newtheorem{Def}[The]{Definition}
\newtheorem{Not}[The]{Notation}
\newtheorem{Claim}[The]{Claim}
\newtheorem{Conc}[The]{Conclusion}
\newtheorem{Ex}[The]{Example}
\newtheorem{Fact}[The]{Fact}
\newtheorem{Formula}[The]{Formula}
\newtheorem{Formulae}[The]{Formulae}
\newtheorem{The-Def}[The]{Theorem and Definition}
\newtheorem{Prop-Def}[The]{Proposition and Definition}
\newtheorem{Lem-Def}[The]{Lemma and Definition}
\newtheorem{Cor-Def}[The]{Corollary and Definition}
\newtheorem{Conc-Def}[The]{Conclusion and Definition}
\newtheorem{Terminology}[The]{Note on terminology}
\newcommand{\C}{\mathbb{C}}
\newcommand{\R}{\mathbb{R}}
\newcommand{\N}{\mathbb{N}}
\newcommand{\Z}{\mathbb{Z}}
\newcommand{\Q}{\mathbb{Q}}
\newcommand{\Proj}{\mathbb{P}}
\newcommand{\Rc}{\mathcal{R}}
\newcommand{\Oc}{\mathcal{O}}
\newcommand{\Vc}{\mathcal{V}}
\newcommand{\Id}{\operatorname{Id}}
\newcommand{\pr}{\operatorname{pr}}
\newcommand{\rk}{\operatorname{rk}}
\newcommand{\del}{\partial}
\newcommand{\delbar}{\bar{\partial}}
\newcommand{\Cdot}{{\raisebox{-0.7ex}[0pt][0pt]{\scalebox{2.0}{$\cdot$}}}}
\newcommand\nilm{\Gamma\backslash G}
\newcommand\frg{{\mathfrak g}}
\newcommand{\fg}{\mathfrak g}
\newcommand{\Oh}{\mathcal{O}}
\newcommand{\Kur}{\operatorname{Kur}}
\newcommand\gc{\frg_\mathbb{C}}
\newcommand\hisashi[1]{{\textcolor{red}{#1}}}
\newcommand\dan[1]{{\textcolor{blue}{#1}}}
\newcommand\luis[1]{{\textcolor{orange}{#1}}}

\begin{center}

{\Large\bf Deformations of Holomorphic $p$-Contact Structures}

\end{center}

\begin{center}

{\large Dan Popovici}

\end{center}

\vspace{1ex}

\noindent{\small{\bf Abstract.} We investigate the notion of holomorphic $p$-contact structure, that we introduced recently in the scalar case with H. Kasuya and L. Ugarte and in the line-bundle case with K. Broder, in the context of a holomorphic family of compact complex manifolds equipped with a holomorphic family of holomorphic line bundles. Firstly, we show that the existence of a holomorphic $p$-contact structure is {\bf deformation-open} when a non-jumping assumption is made on the $\bar\partial$-cohomology of bidegree $(p,\,0)$. This ought to become a source of further examples of holomorphic $p$-contact manifolds. Secondly, we obtain a Berndtsson-type {\bf (semi-)positivity} result for the holomorphic vector bundle over the base induced, as a direct image, by the line bundle on the total space that supports the p-contact structures. Thirdly, we introduce a moduli space of relative holomorphic p-contact structures and we show that it is {\bf Hausdorff} and {\bf open} in a certain quotient topological space.}

\vspace{1ex}

\section{Introduction}\label{section:introd} The following notion was introduced and studied in [KPU25a] and [KPU25b] in the case where the line bundle $L$ is trivial and in [BP25] in the case of an arbitrary holomorphic line bundle $L$.  

\begin{Def}($(1)$ of Definition 1.2 in [BP25])\label{Def:hol_p-contact_bundle} Let $X$ be a compact complex manifold with $\mbox{dim}_\C X =n=2p+1=4l+3$. 

\vspace{1ex}

We say that $X$ is a {\bf (locally) holomorphic $p$-contact manifold} if there exist a holomorphic line bundle $L$ and an {\bf $L$-valued holomorphic $p$-contact structure} on $X$, to wit a form $\Gamma\in C^\infty_{p,\,0}(X,\,L)$ satisfying the following two conditions: \begin{eqnarray}\label{eqn:L_hol-p-contact_Gamma_introd}(a)\,\,\, \bar\partial\Gamma = 0  \hspace{5ex}\mbox{and}\hspace{5ex} (b)\,\,\,\Gamma\wedge\partial\Gamma \neq 0 \hspace{2ex}\mbox{at every point of}\hspace{2ex} X.\end{eqnarray}

If the line bundle $L$ is trivial, we say that the holomorphic $p$-contact structure $\Gamma$ is {\bf scalar-valued}. The line bundle $L$ is called a {\bf $p$-contact bundle}.

\end{Def}

Note that condition (a) in (\ref{eqn:L_hol-p-contact_Gamma_introd}) means that $\Gamma$ is an $L$-valued {\it holomorphic} $(p,\,0)$-form on $X$. In particular, such a $\Gamma$ defines a $\bar\partial$-cohomology class $[\Gamma]_{\bar\partial}\in H^{p,\,0}_{\bar\partial}(X,\,L)$. This class is a singleton since the only $\bar\partial$-exact $(p,\,0)$-form is the zero form. For this reason, we will identify the class $[\Gamma]_{\bar\partial}$ with its unique representative $\Gamma$, so we can write $\Gamma\in H^{p,\,0}_{\bar\partial}(X,\,L)$. As for the non-degeneracy condition (b) in (\ref{eqn:L_hol-p-contact_Gamma_introd}), it was explained in [BP25, Observation 2.1] that, although $\partial\Gamma$ does not make global sense as an $L$-valued $(p+1,\,0)$-form on $X$, the exterior product $\Gamma\wedge\partial\Gamma$ is a well-defined $L^2$-valued $(n,\,0)$-form on $X$ that equals $\Gamma\wedge D_h'\Gamma$ for any $C^\infty$ Hermitian fibre metric $h$ on $L$, where $D_h'$ is the $(1,\,0)$-part of the Chern connection $D_h = D_h' + \bar\partial$ of $(L,\,h)$.

Another observation (see $\S2.4$ of [BP25]) is that, whenever an $L$-valued holomorphic $p$-contact structure $\Gamma\in H^{p,\,0}_{\bar\partial}(X,\,L)$ exists on $X$, the holomorphic line bundle $K_X\otimes L^2$ is trivial (making the dual $L^{-1}$ of $L$ into a {\it spin structure} on $X$ -- spin structures identify with the square roots, if any, of the canonical bundle of $X$). Thus, $\Gamma\wedge\partial\Gamma\in H^{n,\,0}_{\bar\partial}(X,\,L^2)\simeq H^0(X,\,K_X\otimes L^2)\simeq\C$ identifies with a complex number after choosing a trivialisation of $K_X\otimes L^2$. We get a holomorphic map \begin{eqnarray}\label{eqn:gamma-map-homogeneous}H^{p,\,0}_{\bar\partial}(X,\,L)\ni\Gamma\stackrel{T}{\longmapsto}\Gamma\wedge\partial\Gamma\in\C\end{eqnarray} that is homogeneous of degree $2$.

\vspace{2ex}

In the present paper, we investigate families of {\it holomorphic $p$-contact structures} by having the $p$-contact manifold $X$ vary in a holomorphic family $\pi:{\cal X}\longrightarrow B$ of compact complex manifolds $(X_t)_{t\in B}$ and the $p$-contact line bundle $L$ in a holomorphic family $(L_t)_{t\in B}$ of holomorphic line bundles $L_t\longrightarrow X_t$. Specifically, we will consider a holomorphic line bundle ${\cal L}\longrightarrow{\cal X}$ over the total space of the family of manifolds and will set $L_t:={\cal L}_{|X_t}\longrightarrow X_t$, the restriction of ${\cal L}$ to the fibre over $t$, for every $t\in B$. Both ${\cal X}$ and $B$ are complex manifolds and $\pi:{\cal X}\longrightarrow B$ is a proper holomorphic submersion. 

In $\S$\ref{subsection:Hp0_line-bundle}, we observe that under the {\bf non-jumping assumption} (\ref{eqn:non-jumping_assumption}) -- also called hypothesis $(H1)$ -- that requires the constancy of the cohomology numbers $\mbox{dim}_\C H^{p,\,0}_{\bar\partial_t}(X_t,\,L_t)$, the direct image sheaf $\pi_\star(\Omega^p_{{\cal X}/B}\otimes{\cal L})$ is {\it locally free}. We will denote by $E={\cal H}^{p,\,0}({\cal L})$ the associated holomorphic vector bundle over $B$.

In $\S$\ref{subsection:def-open_p-contact}, we prove the first main result of the paper, Theorem \ref{The:deformation-openness_H1}. It can be loosely restated as

\begin{The}\label{The:introd_deformation-openness_H1} Under the {\bf non-jumping assumption} $(H1)$, the existence of an $L_t$-valued holomorphic $p$-contact structure on $X_t$ is a {\bf deformation-open} property.

\end{The}  

By this we mean that, whenever a holomorphic $p$-contact structure $\Gamma_0\in H^{p,\,0}_{\bar\partial_0}(X_0,\,L_0)$ exists on some fibre (say $X_0$, with $0\in B$), such a structure also exists on every sufficiently close nearby fibre $X_t$. Moreover, any such $\Gamma_0\in H^{p,\,0}_{\bar\partial_0}(X_0,\,L_0)$ deforms holomorphically to holomorphic $p$-contact structures $\Gamma_t\in H^{p,\,0}_{\bar\partial_t}(X_t,\,L_t)$ on $X_t$ for all $t\in B$ sufficiently close to $0$.

Without the non-jumping hypothesis $(H1)$, this result is emphatically false, as noticed in [KPU25b, Corollary 3.3] for scalar-valued holomorphic $p$-contact structures. (See also the discussion at the beginning of $\S$\ref{subsection:def-open_p-contact} below.) On the other hand, Theorem \ref{The:introd_deformation-openness_H1} ought to yield new examples of {\it holomorphic $p$-contact manifolds} by deforming known examples.

\vspace{2ex}

In $\S$\ref{section:families_p-contact-bundles}, we shift our perspective and discuss the positivity/negativity properties the vector bundles ${\cal L}\longrightarrow{\cal X}$ and ${\cal H}^{p,\,0}({\cal L})\longrightarrow B$ can or cannot have in the presence of a holomorphic $p$-contact structure.

The results in $\S$\ref{subsection:Hp0_curvature_compact-base} and $\S$\ref{subsection:L_curvature_p-contact} use consequences of results from [Pop26], respectively [BP25], to prove, under appropriate assumptions, that ${\cal H}^{p,\,0}({\cal L})\longrightarrow B$ admits no $C^\infty$ Hermitian fibre metric that makes it {\it quasi-negative} in a certain sense (cf. Proposition \ref{Prop:Hp0_non-negative_p-contact}), respectively that there is no (singular or otherwise) Hermitian fibre metric on ${\cal L}$ that makes it {\it semi-negative} in a certain sense (cf. Corollary \ref{Cor:L_curvature_p-contact}).

Proposition \ref{Prop:Hp0_non-negative_p-contact} and Corollary \ref{Cor:L_curvature_p-contact} are preliminary observations showing that the holomorphic $p$-contact condition is incompatible with certain (semi-)negativity properties of both the holomorphic line bundle ${\cal L}$ over the total space ${\cal X}$ and the holomorphic direct-image vector bundle ${\cal H}^{p,\,0}({\cal L})$ over the base ($B$ or ${\cal Y}$) of the family of manifolds. This naturally leads one to ask whether a Berndtsson-type result (cf. [Ber09, Theorem 1.2]) involving appropriate (semi-)positivity properties of ${\cal L}$ and ${\cal H}^{p,\,0}({\cal L})$ can hold in our $p$-contact setting.

\begin{Question}\label{Question:q-Berndtsson} If ${\cal L}\longrightarrow{\cal X}$ satisfies an appropriate (semi-)positivity condition and is fibrewise $p$-contact, does it follow that ${\cal H}^{p,\,0}({\cal L})\longrightarrow{\cal Y}$ is (semi-)positive in some sense?

\end{Question}  

Note that [Ber09] deals with $L_t$-valued $(n,\,0)$-forms (where $n$ is the fibre dimension), while we are dealing with $L_t$-valued $(p,\,0)$-forms when a $p$-contact structure exists on some (or all) of the fibres $X_t$. Thus, the drop from the top holomorphic degree $n$ to the intermediate degree $p<n$ will have to be compensated for by the $p$-contact hypothesis.

\vspace{2ex}

The second main result of the paper is obtained in $\S$\ref{subsection:E_positivity} as an indirect answer to Question \ref{Question:q-Berndtsson}.

\vspace{1ex}

With any $C^\infty$ Hermitian fibre metric $h$ on ${\cal L}\longrightarrow{\cal X}$ and any Hermitian metric $\omega$ on ${\cal X}$, we associate a $C^\infty$ Hermitian fibre metric $h_E$ on $E={\cal H}^{p,\,0}({\cal L})\longrightarrow B$ defined by the $L^2_{\omega_t,\,h_t}$-inner product $\langle\langle\,\cdot\,,\,\,\cdot\,\rangle\rangle_{\omega_t,\,h_t}$ on each fibre $E_t=H^{p,\,0}_{\bar\partial_t}(X_t,\,L_t)$ with $t\in B$. Similarly, any $h$ on ${\cal L}$ induces a $C^\infty$ Hermitian fibre metric $h_F$ on the holomorphic vector bundle $F={\cal H}^{n,\,0}({\cal L}^2)\longrightarrow B$. It, too, is defined by the $L^2_{h_t^2}$-inner product $\langle\langle\,\cdot\,,\,\,\cdot\,\rangle\rangle_{F_t,\,h_t^2}$ (no need for any $\omega_t$ here) on each fibre $F_t=H^{n,\,0}_{\bar\partial_t}(X_t,\,L_t^2)$ with $t\in B$. Our $(H1)$ and $p$-contact assumptions imply:

\vspace{1ex}

(i)\, the {\it well-definedness} of $F$ (without the need of any extra non-jumping assumption in bidegree $(n,\,0)$ for the line bundles $L_t^2$ -- this is automatic under our hypotheses and, unlike in [Ber09], holds without any K\"ahlerianity assumption on ${\cal X}$);

\vspace{1ex}

(ii)\, the fact that $F$ has {\it rank one}. Actually, $F$ is the vector bundle considered by Berndtsson in [Ber09] (induced here by ${\cal L}^2$ rather than ${\cal L}$), but our $p$-contact setting makes it a {\it line bundle}.

\vspace{1ex}

Another key thing that the holomorphic $p$-contact hypothesis does is to induce a {\it holomorphic symmetric bilinear} bundle map $A:E\times E\longrightarrow F$ (cf. Definition \ref{Def:bilinear-form}) acting as a bridge between the holomorphic vector bundles $E$ and $F$ over $B$. This further enables us to define a notion that we call {\bf A-twisted Griffiths (semi-)positivity} (cf. Definition \ref{Def:A-twisted_Griffiths}) for $(E,\,h_E)\longrightarrow B$. It can be construed as a Griffiths-type (semi-)positivity property of $E$ that is read off the line bundle $F$ after the Chern curvature form of $(E,\,h_E)$ has been transported to $F$ by $A$.

\vspace{1ex}

Our positivity results yield the following

\begin{The}\label{The:introd_F_positivity_Berndtsson-application} Let $\pi:{\cal X}\longrightarrow B$ be a holomorphic family of compact complex manifolds of dimension $n=2p+1=4l+3$ over an arbitrary complex manifold base $B$. Let ${\cal L}\longrightarrow{\cal X}$ be a holomorphic line bundle satisfying assumption $(H1)$ globally on $B$. Suppose there exists an $L_t$-valued holomorphic $p$-contact structure $\Gamma_t\in H^{p,\,0}_{\bar\partial_t}(X_t,\,L_t)$ on $X_t$ for every $t\in B$.

  If ${\cal X}$ admits a {\bf K\"ahler} metric and ${\cal L}\rightarrow{\cal X}$ has a $C^\infty$ Hermitian fibre metric $h$ whose Chern curvature form $i\Theta_h({\cal L})$ is {\bf (semi-)positive} everywhere on ${\cal X}$, then the Hermitian holomorphic vector bundle $(E,\, h_E)$ is {\bf A-twisted Griffiths (semi-)positive}.

\end{The} 

As a matter of fact, this result follows by combining Theorem 1.2 in [Ber09] (restated here, for the reader's convenience, as Theorem \ref{The:Berndtsson} for the line bundle $F$ on $B$ -- this is the only reason why we need to assume K\"ahlerianity on ${\cal X}$, as it is Berndtsson's assumption) with our Theorem \ref{The:F-E_positivity} (which does not require ${\cal X}$ to be K\"ahler). The latter asserts the validity of the implications: \begin{eqnarray*}i\Theta_{h_F}(F)\geq 0 & \implies & (E,\,h_E)\geq 0 \hspace{2ex} (\mbox{\bf A-twisted Griffiths})\\
  i\Theta_{h_F}(F)> 0 & \implies & (E,\,h_E)> 0 \hspace{2ex} (\mbox{\bf A-twisted Griffiths}).\end{eqnarray*} Berndtsson's Theorem \ref{The:Berndtsson} yields the implications: \begin{eqnarray*}i\Theta_h({\cal L})\geq 0 & \implies &  i\Theta_{h_F}(F)\geq 0 \\
 i\Theta_h({\cal L})> 0 & \implies &  i\Theta_{h_F}(F)> 0\end{eqnarray*} when ${\cal X}$ is K\"ahler.

\vspace{2ex}

In $\S$\ref{section:moduli}, we propose, for every positive integer $r$, the notion of {\bf r-th moduli space of relative holomorphic p-contact structures} on $({\cal X},\,B,\,{\cal L})$ (cf. Definition \ref{Def:moduli-space}). It is defined as the topological space, modulo the right action of the unitary group $U_h({\cal L})$ of $C^\infty$ isometries of ${\cal L}$ with respect to a pregiven $C^\infty$ Hermitian fibre metric $h$, of holomorphic structures $D''$ on ${\cal L}$ that satisfy two conditions. The first condition requires $D''$ to satisfy a strengthened version of hypothesis $(H1)$, called hypothesis $(H_1^{(r)})$, in which the dimensions of the $\C$-vector spaces $H^{p,\,0}_{D''_t}(X_t,\,L_t)$ equal $r$ for all $t\in B$. The second condition requires the existence of a {\it continuous $p$-contact section} of the vector bundle ${\cal H}^{p,\,0}_{D''}({\cal L})\longrightarrow B$ on a neighbourhood of every point of $B$. (Cf. Definition \ref{Def:C_p_set}.)

Our main result in this section is Theorem \ref{The:moduli-space_prop}. It can be loosely restated as

\begin{The}\label{The:introd_moduli-space_prop} The moduli space ${\cal M}^{(r)}_p({\cal X},\,B,\,{\cal L})$ is {\bf Hausdorff} and, if $B$ is {\bf compact}, it is {\bf open} in the quotient topological space ${\cal P}^{(r)}_p({\cal X},\,B,\,{\cal L})/U_h({\cal L})$, where ${\cal P}^{(r)}_p({\cal X},\,B,\,{\cal L})$ consists of the holomorphic structures $D''$ on ${\cal L}$ that satisfy hypothesis $(H_1^{(r)})$.

\end{The}

This naturally invites further questions as to the structure of the moduli space ${\cal M}^{(r)}_p({\cal X},\,B,\,{\cal L})$. For example, is it a {\it complex analytic space}? Is it a {\it complex manifold} in the neighbourhood of certain points? We hope to be able to take up these questions in future work.

\section{Families of holomorphic $p$-contact bundles}\label{section:families_p-contact-bundles}

We first set the stage in $\S$\ref{subsection:Hp0_line-bundle} before giving the main result of this section in $\S$\ref{subsection:def-open_p-contact}.

\subsection{The holomorphic vector bundle ${\cal H}^{p,\,0}({\cal L})$ over $B$}\label{subsection:Hp0_line-bundle}

Let $\pi:{\cal X}\longrightarrow B$ be a {\it proper holomorphic submersion}, where ${\cal X}$ is a complex manifold and $B\subset\C^N$ is an open ball centred at $0$, for some positive integer $N$. The fibres $X_t:=\pi^{-1}(t)$, with $t\in B$, of this holomorphic family of compact complex manifolds are supposed to have complex dimension $n=2p+1$ for some {\it odd} positive integer $p$. (This assumption is not needed in this $\S$\ref{subsection:Hp0_line-bundle}, but it will be in subsequent sections in relation to holomorphic $p$-contact structures.)

Let ${\cal L}\longrightarrow{\cal X}$ be a holomorphic line bundle over the total space of the family. For $t\in B$, we denote by $L_t:={\cal L}_{|X_t}\longrightarrow X_t$ the restriction of ${\cal L}$ to the fibre above $t$. It is a holomorphic line bundle on $X_t$.  

We will make the following {\bf non-jumping assumption} on the dimensions of the $L_t$-valued $\bar\partial_t$-cohomology spaces of bidegree $(p,\,0)$: \begin{eqnarray}\label{eqn:non-jumping_assumption} (H1)\hspace{3ex} \mbox{dim}_\C H^{p,\,0}_{\bar\partial_t}(X_t,\,L_t) \hspace{3ex}\mbox{is independent of}\hspace{1ex} t\in B,\end{eqnarray} where $\bar\partial_t$ is the canonical $\bar\partial$-operator on $L_t$, defining the holomorphic structure of $L_t$, extending to $L_t$-valued forms the $\bar\partial$-operator $\bar\partial_t$ that defines the complex structure of the compact complex manifold $X_t$.

Since $\pi$ is a holomorphic submersion, the sheaf $\Omega^p_{{\cal X}/B}$ of germs of relative holomorphic $(p,\,0)$-forms on ${\cal X}$ is locally free. We will denote by the same symbol the corresponding holomorphic vector bundle and by $\widetilde{\cal E}:=\Omega^p_{{\cal X}/B}\otimes{\cal L}$ the holomorphic vector bundle on ${\cal X}$ obtained by twisting it with ${\cal L}$. The non-jumping hypothesis $(H1)$ ensures, thanks to Grauert's {\it direct image theorem}, that the direct image sheaf \begin{eqnarray}\label{eqn:Hp0-L_def}{\cal H}^{p,\,0}({\cal L}):=\pi_\star\bigg(\Omega^p_{{\cal X}/B}\otimes{\cal L}\bigg) = \pi_\star\widetilde{\cal E}\end{eqnarray} is {\bf locally free}. We will denote by the same symbol the corresponding holomorphic vector bundle on $B$. The restrictions of $\widetilde{\cal E}$ to the fibres of the manifold family $\pi:{\cal X}\longrightarrow B$ are $\widetilde{\cal E}_{|X_t} = \Omega^p_{X_t}\otimes L_t$ and there is a canonical base-change isomorphism \begin{eqnarray*}(\pi_\star\widetilde{\cal E})_t\otimes_{{\cal O}_{B,\,t}}\C\stackrel{\simeq}{\longrightarrow} H^0(X_t,\,\widetilde{\cal E}_{|X_t}),  \hspace{5ex} t\in B.\end{eqnarray*} So the fibres of the vector bundle (associated with the sheaf) ${\cal H}^{p,\,0}({\cal L})=\pi_\star\widetilde{\cal E}$ are \begin{eqnarray*}{\cal H}^{p,\,0}({\cal L})_t=(\pi_\star\widetilde{\cal E})_t\otimes_{{\cal O}_{B,\,t}}\C \simeq H^0(X_t,\,\Omega^p_{X_t}\otimes L_t) = H^{p,\,0}_{\bar\partial_t}(X_t,\,L_t),  \hspace{5ex} t\in B.\end{eqnarray*} 

\vspace{2ex}

The weaker conclusion that, under hypothesis $(H1)$, ${\cal H}^{p,\,0}({\cal L})$ is a $C^\infty$ complex vector bundle over $B$ whose fibres are the above vector spaces can be obtained by applying the classical Kodaira-Spencer theory [KS60, Theorems 1-5]. Indeed, let $h$ be a $C^\infty$ Hermitian fibre metric on ${\cal L}$. For every $t\in B$, let $h_t$ be the restriction of $h$ to $L_t$ and let $\Delta_t'':=\bar\partial_t\bar\partial_t^\star + \bar\partial_t^\star\bar\partial_t:C^\infty_{r,\,s}(X_t,\,L_t)\longrightarrow C^\infty_{r,\,s}(X_t,\,L_t)$ be the induced $\bar\partial$-Laplacian on smooth $L_t$-valued forms of any bidegree $(r,\,s)$, where $\bar\partial_t^\star$ is the adjoint of $\bar\partial_t$ with respect to the $L^2$-inner product defined by $h_t$. Thus, $(\Delta_t'')_{t\in B}$ is a $C^\infty$ family of elliptic differential operators of order two. Since, for each $t$, $H^{p,\,0}_{\bar\partial_t}(X_t,\,L_t)$ is isomorphic (under the Hodge isomorphism) to the kernel of $\Delta_t'':C^\infty_{p,\,0}(X_t,\,L_t)\longrightarrow C^\infty_{p,\,0}(X_t,\,L_t)$, hypothesis $(H1)$ ensures that the dimension of the kernel of $\Delta_t''$ in bidegree $(p,\,0)$ is independent of $t\in B$. Thus, the Kodaira-Spencer theory implies the existence of a $C^\infty$ complex vector bundle over $B$ whose fibres are these kernels. The same then goes for the cohomology spaces $H^{p,\,0}_{\bar\partial_t}(X_t,\,L_t)$ with $t\in B$.

\vspace{2ex}

Recall that a similar holomorphic vector bundle, denoted there by $E$ (which would correspond in our notation to ${\cal H}^{n,\,0}({\cal L})$), was considered by Berndtsson as the setting of his Theorem 1.2 in [Ber09]. His bidegree $(n,\,0)$ is one key difference to our bidegree $(p,\,0)$ case. Moreover, he assumes the manifold ${\cal X}$ to be {\it K\"ahler} and this implies, thanks to a further semi-positivity assumption on ${\cal L}$ and to the Ohsawa-Takegoshi $L^2$ extension theorem, the analogue of our non-jumping condition $(H1)$: $\mbox{dim}_\C H^{n,\,0}_{\bar\partial_t}(X_t,\,L_t)$ is independent of $t$. (Recall that this dimension always depends upper semi-continuously on $t$, so the possibility of holomorphically extending any element of any given $H^{n,\,0}_{\bar\partial_{t_0}}(X_{t_0},\,L_{t_0})$ to the nearby spaces $H^{n,\,0}_{\bar\partial_t}(X_t,\,L_t)$ ensures dimension constancy.) When ${\cal L}$ is trivial, the non-jumping of each $h^{p,\,q}(X_t,\,\C)$ is implied by the far weaker assumption that the Fr\"olicher spectral sequence of each fibre $X_t$ degenerate at $E_1$. However, there are at least two peculiarities that make these observations inapplicable in our case:

\vspace{1ex}

(a)\, no version of the Ohsawa-Takegoshi $L^2$ extension theorem is currently known in any intermediate bidegree $(p,\,0)$ with $1\leq p<n$, even under the K\"ahler assumption;

\vspace{1ex}

(b)\, no scalar-valued holomorphic $p$-contact structure can exist on a compact complex manifold whose Fr\"olicher spectral sequence degenerates at $E_1$, as noticed in [KPU25a] -- we will later assume the existence of such a structure on at least one fibre $X_t$.

\vspace{1ex}

In other words, our non-jumping assumption $(H1)$ in bidegree $(p,\,0)$ ensures the existence of the holomorphic vector bundle ${\cal H}^{p,\,0}({\cal L})$ without any K\"ahlerianity (or Fr\"olicher degeneration) assumption on ${\cal X}$ or $X_t$ and without any (semi-)positivity assumption on ${\cal L}$.

\subsection{Deformation openness for holomorphic $p$-contact structures}\label{subsection:def-open_p-contact}

It was observed in [KPU25b, Corollary 3.3] that, for compact complex manifolds, the existence of holomorphic contact structures is not a deformation open property. The counter-example to deformation openness obtained there is provided by the Kuranishi family $(X_t)_{t\in B}$ of the {\it Iwasawa manifold} $X_0 = I^{(3)}$, a well-known compact complex $3$-dimensional manifold that admits {\it scalar-valued holomorphic contact structures}, the classical notion that coincides with our notion of {\it scalar-valued holomorphic $1$-contact structures} in complex dimension $3$. It was proved in Corollary 3.3 of [KPU25b] that {\bf no} small deformation $X_t$ belonging to either of the Nakamura classes (ii) or (iii) admits scalar-valued holomorphic contact structures. A key feature of that example (which was not exploited there) is that $h^{1,\,0}_{\bar\partial}(X_0) = 3$, while $h^{1,\,0}_{\bar\partial}(X_t) = 2$ for every $t\neq 0$ with $X_t$ in the Nakamura class  (ii) or (iii).

Obviously, when $n=3$, $p=1$. The starting point of the present work is the observation that the existence of an $L$-valued holomorphic $p$-contact structure becomes a {\it deformationn open} property if the cohomology dimension $h^{p,\,0}_{\bar\partial_t}(X_t,\,L_t) :=\mbox{dim}_\C H^{p,\,0}_{\bar\partial_t}(X_t,\,L_t)$ is independent of $t$. In other words, the non-jumping assumption $(H1)$ forces deformation openness.

\begin{The}\label{The:deformation-openness_H1} Let $\pi:{\cal X}\longrightarrow B$ be a holomorphic family of compact complex manifolds of dimension $n=2p+1=4l+3$, where $B\subset\C^N$ is an open ball centred at $0$ for some positive integer $N$. Let ${\cal L}\longrightarrow{\cal X}$ be a holomorphic line bundle satisfying assumption $(H1)$.

  Suppose there exists an {\bf $L_0$-valued holomorphic $p$-contact structure} $\Gamma_0\in H^{p,\,0}_{\bar\partial_0}(X_0,\,L_0)$ on $X_0$. Then, after possibly shrinking $B$ about $0$, there exists a holomorphic family $(\Gamma_t)_{t\in B}$ of {\bf $L_t$-valued holomorphic $p$-contact structures} $\Gamma_t\in H^{p,\,0}_{\bar\partial_t}(X_t,\,L_t)$ on $X_t$ whose member for $t=0$ is the given $\Gamma_0$.

\end{The}

\noindent {\it Proof.} By hypothesis, $\Gamma_0\in{\cal H}^{p,\,0}({\cal L})_0$ is a non-zero element in the fibre at $0\in B$ of the holomorphic vector bundle ${\cal H}^{p,\,0}({\cal L})$ over $B$ (introduced in $\S$\ref{subsection:Hp0_line-bundle}). Thus, $\Gamma_0$ extends (non-uniquely and non-canonically) to a non-vanishing holomorphic section $\Gamma$ of this vector bundle on a small neighbourhood of $0$ in $B$. After possibly shrinking $B$, we may assume this neighbourhood to be $B$. We get $\Gamma\in H^0(B,\,{\cal H}^{p,\,0}({\cal L}))$ with $\Gamma(0) = \Gamma_0$. For every $t\in B$, we set \begin{eqnarray*}\Gamma_t:=\Gamma(t)\in{\cal H}^{p,\,0}({\cal L})_t = H^{p,\,0}_{\bar\partial_t}(X_t,\,L_t).\end{eqnarray*} By construction, after identifying each $\Gamma_t$ with its (singleton) $\bar\partial_t$-cohomology class, we have $\bar\partial_t\Gamma_t = 0$ for all $t\in B$.

We have thus got a $C^\infty$ (hence also continuous) map \begin{eqnarray}\label{eqn:map-identification_Gamma_Ehresmann}{\cal X}\simeq B\times X\ni(t,\,z)\longmapsto\Gamma_t(z)\wedge(\partial_t\Gamma_t)(z),\end{eqnarray} where $t\in B$ is the base variable and $z\in X_0\simeq X$ is the fibre (vertical) variable obtained after identifying each fibre $X_t$ with a fixed $C^\infty$ manifold $X$. (Recall the classical Ehresmann theorem: the holomorphic family $\pi:{\cal X}\longrightarrow B$ of compact complex manifolds is necessarily $C^\infty$ trivial over the simply connected base $B$, so we have a $C^\infty$ isomorphism ${\cal X}\simeq B\times X$, where $X$ is the $C^\infty$ manifold underlying the fibres $X_t$.)

Now, for every point $x_0\in X$, $\Gamma_0(x_0)\wedge(\partial_0\Gamma_0)(x_0)$ is non-zero. Hence, by continuity, there are neighbourhoods $U_{x_0}\subset X$ of $x_0\in X$ and $V_{x_0}\subset B$ of $0\in B$ such that $\Gamma_t(z)\wedge(\partial_t\Gamma_t)(z)$ remains non-zero for every $(t,\,z)\in V_{x_0}\times U_{x_0}$. Since $X$ is compact, finitely many such open subsets $U_{x_0}$ cover $X$. Taking the intersection $B_0$ of the (finitely many) corresponding open neighbourhoods $V_{x_0}$ of $0$ in $B$, we get \begin{eqnarray*}\Gamma_t\wedge\partial_t\Gamma_t\neq 0 \hspace{3ex} \mbox{everywhere on}\hspace{1ex} X_t, \hspace{3ex} t\in B_0.\end{eqnarray*}

This proves that $\Gamma_t$ is an $L_t$-valued holomorphic $p$-contact structure on $X_t$ for every $t\in B_0$. \hfill $\Box$

\vspace{2ex}

Recall that a large number of examples of holomorphic $p$-contact manifolds with scalar-valued $p$-contact structures were exhibited in [KPU25a] and [KPU25b]. However, few examples of holomorphic $p$-contact manifolds with $L$-valued $p$-contact structures for non-trivial holomorphic line bundles $L$ were observed in [BP25] and are currently known. Theorem \ref{The:deformation-openness_H1} goes some distance towards producing new examples of such manifolds as small deformations of existing ones when assumption $(H1)$ is satisfied. It thus turns deformation families into sources of examples.

\vspace{2ex}

We now give an application to the structure of the holomorphic vector bundle $E:={\cal H}^{p,\,0}({\cal L})\longrightarrow B$. We start with a general observation about holomorphic $p$-contact bundles.

Let $L$ be a holomorphic line bundle over a compact complex manifold $X$ with $\mbox{dim}_\C X =n = 2p+1$ and $p$ odd. If an $L$-valued holomorphic $p$-contact structure $\Gamma\in H^{p,\,0}_{\bar\partial}(X,\,L)$ exists, we saw in the introduction that the holomorphic map (\ref{eqn:gamma-map-homogeneous}), which is homogeneous of degree $2$, can be viewed as assuming values in $\C$. Thus, the set \begin{eqnarray*}H^{p,\,0}_{\bar\partial}(X,\,L)_{p\mbox{-con}}:=\bigg\{\Gamma\in H^{p,\,0}_{\bar\partial}(X,\,L)\,\mid\,T(\Gamma)\neq 0\bigg\}\subset H^{p,\,0}_{\bar\partial}(X,\,L)\end{eqnarray*} of $L$-valued holomorphic $p$-contact structures on $X$ is a non-empty {\it Zariski-open} subset of $H^{p,\,0}_{\bar\partial}(X,\,L)$.

\begin{Cor}\label{Cor:p-contact_basis} If an $L$-valued holomorphic $p$-contact structure exists on $X$, the $\C$-vector space $H^{p,\,0}_{\bar\partial}(X,\,L)$ admits a $\C$-basis consisting solely of holomorphic $p$-contact forms $\Gamma_1,\dots, \Gamma_r$, where $r:=h^{p,\,0}_{\bar\partial}(X,\,L)$ is the dimension of $H^{p,\,0}_{\bar\partial}(X,\,L)$. We will call such a basis a {\bf $p$-contact basis}.

\end{Cor}  

\noindent {\it Proof.} We proceed by induction: suppose that $\Gamma_1,\dots, \Gamma_k\in H^{p,\,0}_{\bar\partial}(X,\,L)_{p\mbox{-con}}$ have already been constructed and are $\C$-linearly independent. If their span $\C\,\langle\Gamma_1,\dots, \Gamma_k\rangle$ is strictly contained in $H^{p,\,0}_{\bar\partial}(X,\,L)$, it is Lebesgue-negligible, so it cannot contain the full-measure subset $H^{p,\,0}_{\bar\partial}(X,\,L)_{p\mbox{-con}}$ of $H^{p,\,0}_{\bar\partial}(X,\,L)$. Consequently, there exists $\Gamma_{k+1}\in H^{p,\,0}_{\bar\partial}(X,\,L)_{p\mbox{-con}}\setminus\C\,\langle\Gamma_1,\dots, \Gamma_k\rangle$. \hfill $\Box$

\begin{Cor}\label{Cor:p-contact_frame} The setting is the same as in Theorem \ref{The:deformation-openness_H1}. In particular, ${\cal L}\longrightarrow{\cal X}$ is supposed to satisfy  assumption $(H1)$ and an {\bf $L_0$-valued holomorphic $p$-contact structure} $\Gamma_0\in H^{p,\,0}_{\bar\partial_0}(X_0,\,L_0)$ is supposed to exist on $X_0$. Let $r$ be the rank of the holomorphic vector bundle $E:={\cal H}^{p,\,0}({\cal L})\longrightarrow B$. Fix an arbitrary $p$-contact basis $\{\Gamma_1^0,\dots, \Gamma_r^0\}$ of $E_0=H^{p,\,0}_{\bar\partial_0}(X_0,\,L_0)$ (as given by Corollary \ref{Cor:p-contact_basis}).

  Then, after possibly shrinking $B$ about $0$, there exists a holomorphic frame $\Gamma_1,\dots , \Gamma_r\in H^0(B,\,E)$ of $E$ such that:

\vspace{1ex}

(i)\, $\Gamma_j(0) = \Gamma_j^0$ for all $j=1,\dots , r$;

\vspace{1ex}

(ii)\, $\Gamma_j(t)\in E_t=H^{p,\,0}_{\bar\partial_t}(X_t,\,L_t)$ is an $L_t$-valued holomorphic {\bf $p$-contact} form on $X_t$ for all $t\in B$ and all $j=1,\dots , r$.

\vspace{1ex}

Such a frame $\{\Gamma_1,\dots , \Gamma_r\}$ will be called a {\bf holomorphic $p$-contact frame} of $E$.

\end{Cor}  

\noindent {\it Proof.} It suffices to apply Theorem \ref{The:deformation-openness_H1} to each $\Gamma_j^0$ of the $p$-contact basis of $E_0$ given by Corollary \ref{Cor:p-contact_basis}. The set $\{\Gamma_1(t),\dots , \Gamma_r(t)\}$ remains a $\C$-basis of $E_t$ for every $t\in B$ if $B$ has been further shrunk enough about $0$.   \hfill $\Box$

\section{Curvature of relative holomorphic $p$-contact bundles}\label{section:curvature_families_p-contact-bundles}

We start with a few generalities on the holomorphic vector bundle $E:={\cal H}^{p,\,0}({\cal L})\longrightarrow B$ constructed in $\S$\ref{subsection:Hp0_line-bundle} under the assumption $(H1)$.

Let $\omega$ be a Hermitian metric on ${\cal X}$ and let $h$ be a $C^\infty$ Hermitian fibre metric on ${\cal L}\longrightarrow{\cal X}$. Then, for every $t\in B$, the restrictions $\omega_t:=\omega_{|X_t}$ and $h_t:=h_{|L_t}$ are a Hermitian metric on $X_t$, respectively a $C^\infty$ Hermitian fibre metric on $L_t$. Recall that we set $n=2p+1 = \mbox{dim}_\C X_t$ for every $t\in B$.

\begin{Lem}\label{Lem:coupling_inner-prod_X} Let $(L,\,h)$ be a Hermitian holomorphic line bundle over an $n$-dimensional complex Hermitian manifold $(X,\,\omega)$. Let $p\in\{0,\dots , n\}$. Then, for  all $\Gamma,\Omega\in C^\infty_{p,\,0}(X,\,L)$, the following identity of real scalar-valued forms holds everywhere on $X$: \begin{eqnarray}\label{eqn:coupling_inner-prod_X}i^{p^2}\,\{\Gamma,\, \Omega\}_h\wedge\omega_{n-p} = \langle\Gamma,\, \Omega\rangle_{\omega,\,h}\,dV_{\omega}.\end{eqnarray}

\end{Lem}

\noindent {\it Proof.} In local holomorphic coordinates $z_1,\dots , z_n$ of $X$ and a local holomorphic frame $\{e\}$ of $L$, we have $\Gamma = \sum_{|I|=p}\Gamma_I\,dz_I\otimes e$ and $\Omega = \sum_{|I|=p}\Omega_I\,dz_I\otimes e$ for some $C^\infty$ locally defined functions $\Gamma_I, \Omega_I$.

We get: \begin{eqnarray*}i^{p^2}\,\{\Gamma,\, \Omega\}_h = \sum\limits_{|I| = |J| =p}i^{p^2}\,\Gamma_I\,\overline\Omega_J\,dz_I\wedge d\bar{z}_J\,e^{-\varphi},\end{eqnarray*} where $\varphi$ is the local weight of the fibre metric $h$ in the chosen local trivialisation of $L$ (i.e. $|e|_h^2 = e^{-\varphi}$).

Now, fix a point $x_0\in X$ and choose the coordinates $z_1,\dots , z_n$ centred at $x_0$ such that $\omega$ is given by the identity matrix at $x_0$ in these coordinates (i.e. $\omega(x_0) = \sum_{1\leq j\leq n} idz_j\wedge d\bar{z}_j$). At $x_0$, we get: \begin{eqnarray*}i^{p^2}\,\{\Gamma,\, \Omega\}_h\wedge\omega_{n-p} = \bigg(\sum_{|I|=p}\Gamma_I\,\overline\Omega_I\,e^{-\varphi}\bigg)\,dV_n = \langle\Gamma,\, \Omega\rangle_{\omega,\,h}\,dV_n,\end{eqnarray*} where $dV_n:=idz_1\wedge d\bar{z}_1\wedge\dots\wedge idz_n\wedge d\bar{z}_n$. Since $dV_n = dV_\omega$ at $x_0$, (\ref{eqn:coupling_inner-prod_X}) follows. \hfill $\Box$

\begin{Rem}\label{Rem:L2_inner-product_p=n} Note that, when $p=n$, the left-hand side of (\ref{eqn:coupling_inner-prod_X}) is independent of $\omega$. Hence, so is the right-hand side, despite its ostensible dependence on $\omega$. 

\end{Rem}

\vspace{2ex}

Applying Lemma \ref{Lem:coupling_inner-prod_X} in our setting to each fibre $(L_t,\,h_t)\longrightarrow(X_t,\,\omega_t)$, we get the following expression for the $L^2$-inner product on $C^\infty_{p,\,0}(X_t,\,L_t)$ induced by the metrics $(\omega_t,\,h_t)$: \begin{eqnarray}\label{eqn:L2_inner-product_expression}\langle\langle\Gamma_t,\, \Omega_t\rangle\rangle_{\omega_t,\,h_t}:=\int\limits_{X_t}\langle\Gamma_t,\, \Omega_t\rangle_{\omega_t,\,h_t}\,dV_{\omega_t} = \int\limits_{X_t} i^{p^2}\,\{\Gamma_t,\, \Omega_t\}_{h_t}\wedge(\omega_t)_{n-p}, \hspace{5ex} t\in B,\end{eqnarray} for all $\Gamma_t, \Omega_t\in C^\infty_{p,\,0}(X_t,\,L_t)$.

\begin{Def}\label{Def:h_E} Let $(\omega,\,h)$ be the pair of a Hermitian metric on ${\cal X}$ and a $C^\infty$ Hermitian fibre metric on ${\cal L}\longrightarrow{\cal X}$. We associate with it the $C^\infty$ Hermitian fibre metric $h_E$ on $E:={\cal H}^{p,\,0}({\cal L})\longrightarrow B$ defined by $B\ni t\longmapsto\langle\langle\,\cdot\,,\, \,\cdot\,\rangle\rangle_{\omega_t,\,h_t}$, where $\langle\langle\,\cdot\,,\, \,\cdot\,\rangle\rangle_{\omega_t,\,h_t}$ is the $L^2_{\omega_t,\,h_t}$-inner product of (\ref{eqn:L2_inner-product_expression}).

\end{Def}

\vspace{2ex}

In our $p$-contact setting, once we have fixed a pair of metrics $(\omega,\,h)$ as in Definition \ref{Def:h_E}, we can reinforce Corollary \ref{Cor:p-contact_frame}.

\begin{Cor}\label{Cor:p-contact-parallel_frame} The setting is the same as in Corollary \ref{Cor:p-contact_frame}. We fix, in addition, a pair of metrics $(\omega,\,h)$ as in Definition \ref{Def:h_E}. Let $D_E = D_E' + D_E''$ be the Chern connection of $(E,\,h_E)$, where $h_E$ is the $C^\infty$ Hermitian fibre metric on $E$ induced by $(\omega,\,h)$.

  Then, after possibly shrinking $B$ about $0$, the {\bf holomorphic $p$-contact frame} $\{\Gamma_1,\dots , \Gamma_r\}$ of $E$ given by Corollary \ref{Cor:p-contact_frame} can be chosen such that, in addition, it is {\bf $D_E$-parallel} at $t=0$, namely \begin{eqnarray}\label{eqn:D_E_paralell}(D_E'\Gamma_j)(0) = 0, \hspace{5ex} j=1,\dots , r.\end{eqnarray}
  
\end{Cor}

\noindent {\it Proof.} Since each $\Gamma_j$ is holomorphic, we automatically have $D_E''\Gamma_j = 0$ everywhere on $B$ for every $j=1,\dots , r$. Thus, once we have proved (\ref{eqn:D_E_paralell}), we will have $(D_E\Gamma_j)(0) = 0$ for all $j$.

Let $\{\Gamma_1,\dots , \Gamma_r\}$ be the holomorphic $p$-contact frame of $E$ 
given by Corollary \ref{Cor:p-contact_frame}. Then, there exist $C^\infty$ $\C$-valued $(1,\,0)$-forms $\theta_{kj}$ on $B$ such that \begin{eqnarray*}D_E'\Gamma_j = \sum\limits_{k=1}^r\theta_{kj}\wedge\Gamma_k  \hspace{5ex} \mbox{on}\hspace{2ex} B.\end{eqnarray*} Set $\theta:=(\theta_{kj})_{1\leq j,k\leq r}$.

Consider a {\it holomorphic} matrix-valued function $M = (M_{kj})_{1\leq k,j\leq r}:B\longrightarrow GL(r,\,\C)$ such that \begin{eqnarray*}M(0) = \mbox{Id}_r \hspace{3ex}\mbox{and}\hspace{3ex} (\partial M)(0) = - \theta(0).\end{eqnarray*} Set $\widetilde\Gamma_j:=\sum\limits_{k=1}^r\Gamma_k\,M_{kj}$ for $j=1,\dots , r$. Then $\widetilde\Gamma_j(0) = \Gamma_j(0)$ for $j=1,\dots , r$.

On the other hand, from $D_E'\widetilde\Gamma_j = \sum\limits_{k=1}^r (D_E'\Gamma_k)\,M_{kj} + \sum\limits_{k=1}^r \Gamma_k\wedge \partial M_{kj}$ (with $j=1,\dots , r$), we get: \begin{eqnarray*}(D_E'\widetilde\Gamma_j)(0) = (D_E'\Gamma_j)(0) - \sum\limits_{k=1}^r \Gamma_k(0)\wedge\theta_{kj}(0) = (D_E'\Gamma_j)(0) - (D_E'\Gamma_j)(0) = 0, \hspace{5ex} j=1,\dots , r.\end{eqnarray*} Note that each $\Gamma_k$ is viewed as a section (i.e. a $(0,\,0)$-form) of $E$ over $B$, not as an $L_t$-valued $(p,\,0)$-form on $X_t$, so $\Gamma_k(0)\wedge\theta_{kj}(0) = \theta_{kj}(0)\wedge\Gamma_k(0)$ and the symbol $\wedge$ can even be suppressed here.

Thus, after replacing the $\Gamma_j$'s by the $\widetilde\Gamma_j$'s, we get (\ref{eqn:D_E_paralell}). Moreover, since $\widetilde\Gamma_j(0) = \Gamma_j(0)$ for every $j$, since $\{\Gamma_1(0),\dots , \Gamma_r(0)\}$ is a $p$-contact basis of $E_0$ and since each $\widetilde\Gamma_j$ varies holomorphically (hence continuously) with $t\in B$, the family $\{\widetilde\Gamma_1(t),\dots , \widetilde\Gamma_r(t)\}$ remains a $p$-contact basis of $E_t$ for all $t$ sufficiently close to $0$.    \hfill $\Box$

\subsection{Curvature of ${\cal H}^{p,\,0}({\cal L})$ over a compact base}\label{subsection:Hp0_curvature_compact-base}

A corollary of the Kobayashi-style vanishing theorem 2.3 in [Pop26] is the following

\begin{Prop}\label{Prop:Hp0_non-negative_p-contact} Let $\pi:{\cal X}\longrightarrow{\cal Y}$ be a proper holomorphic submersion between complex manifolds and let ${\cal L}$ be a holomorphic line bundle over ${\cal X}$. Suppose that:

\vspace{1ex}

(a)\, ${\cal Y}$ is {\bf compact} of complex dimension $N$;

\vspace{1ex}

(b)\, $\mbox{dim}_\C H^{p,\,0}_{\bar\partial_t}(X_t,\,L_t)$ is independent of $t\in{\cal Y}$;

\vspace{1ex}

(c)\, there exists a non-zero global holomorphic section $\Gamma\in H^0({\cal Y},\,{\cal H}^{p,\,0}({\cal L}))\setminus\{0\}$. 

\vspace{1ex}

Then, there exists {\bf no} $C^\infty$ Hermitian fibre metric $h$ on ${\cal H}^{p,\,0}({\cal L})$ whose Chern curvature form $i\Theta_h({\cal H}^{p,\,0}({\cal L}))$ satisfies, for some Gauduchon metric $\rho$ on ${\cal Y}$:

\vspace{1ex}

(i)\, the inequality: \begin{eqnarray}\label{eqn:negarive-curvature_Hp0L}\bigg\{\bigg(i\Theta_h({\cal H}^{p,\,0}({\cal L}))\wedge\rho^{N-1}\bigg)\,u,\,u\bigg\}_h\leq 0\end{eqnarray} for every $y\in{\cal Y}$ and every $u\in {\cal H}^{p,\,0}({\cal L})_y$;

\vspace{1ex}

(ii)\, and the strict inequality $\{(i\Theta_h({\cal H}^{p,\,0}({\cal L}))\wedge\rho^{N-1})\,u,\,u\}_h< 0$ for some $y_0\in{\cal Y}$ and every $u\in {\cal H}^{p,\,0}({\cal L})_{y_0}\setminus\{0\}$.

\end{Prop}

\vspace{2ex}

Note that hypothesis (c) is satisfied if we make instead the following hypothesis: \begin{eqnarray}\label{hypothesis:Hp0_non-negative_p-contact_hyp_c'}\nonumber & &  \mbox{\it there exists a global holomorphic section $\Gamma\in H^0({\cal Y},\,{\cal H}^{p,\,0}({\cal L}))$ and a point $y_0\in{\cal Y}$}\\
\nonumber (c') & &  \mbox{\it such that $\Gamma(y_0)\in H^{p,\,0}_{\bar\partial_{y_0}}(X_{y_0},\,L_{y_0})$ is an {\bf $L_{y_0}$-valued holomorphic $p$-contact structure}} \\
  & & \mbox{\it on $X_{y_0}:=\pi^{-1}(y_0)\subset{\cal X}$, where $L_{y_0}:= {\cal L}_{|X_{y_0}}$ and $p$ is odd with $\mbox{dim}_\C X_{y_0} = n = 2p+1$}.    \end{eqnarray}

\noindent {\it Proof of Proposition \ref{Prop:Hp0_non-negative_p-contact}.} It is very similar to the proof of Theorem 2.3 in [Pop26]. For the convenience of the reader, we spell out the details. We reason by contradiction. Suppose there exists a $C^\infty$ Hermitian fibre metric $h$ on ${\cal H}^{p,\,0}({\cal L})$  whose Chern curvature form satisfies the property in the statement. Let $D_h = D_h' + \bar\partial$ be the Chern connection of $({\cal H}^{p,\,0}({\cal L}),\,h)$.

Direct computations yield: \begin{eqnarray*}\partial\{\Gamma,\,\Gamma\}_h = \{D_h'\Gamma,\,\Gamma\}_h + \{\Gamma,\,\bar\partial \Gamma\}_h = \{D_h'\Gamma,\,\Gamma\}_h,\end{eqnarray*} since $\bar\partial \Gamma = 0$. Taking $\bar\partial$, we further get: \begin{eqnarray*}\bar\partial\partial\{\Gamma,\,\Gamma\}_h = \{\bar\partial D_h'\Gamma,\,\Gamma\}_h - \{D_h'\Gamma,\,D_h'\Gamma\}_h = \{\Theta_h({\cal H}^{p,\,0}({\cal L}))\Gamma,\,\Gamma\}_h - \{D_h'\Gamma,\,D_h'\Gamma\}_h,\end{eqnarray*} where we used again the fact that $\bar\partial \Gamma = 0$ to get the curvature term. Hence: \begin{eqnarray*}(N-1)\,\bigg(\Delta_{\rho}|\Gamma|^2_h\bigg)\,dV_\rho & = & i\bar\partial\partial\{\Gamma,\,\Gamma\}_h\wedge\rho^{N-1} \\
  & = & \bigg\{\bigg(i\Theta_h({\cal H}^{p,\,0}({\cal L}))\wedge\rho^{N-1}\bigg)\Gamma,\,\Gamma\bigg\}_h - i\,\{D_h'\Gamma,\,D_h'\Gamma\}_h\wedge\rho^{N-1} \leq 0\end{eqnarray*} at every point of ${\cal Y}$. By $\Delta_{\rho}:=\Lambda_{\rho}(i\bar\partial\partial)$ we denote the standard Laplacian on functions induced by $\rho$. We have also used the fact that, given any $N$-dimensional complex manifold $X$ and any Hermitian holomorphic vector bundle $(E,\,h)\longrightarrow X$, for every $\eta\in C^\infty_{1,\,0}(X,\,E)$, the $\R$-valued $C^\infty$ $(N,\,N)$-form $i\,\{\eta,\,\eta\}_h\wedge\rho^{N-1}$ is $\geq 0$ at every point of $X$. (See e.g.  Lemma 2.9. of [Pop25].)

  Since $\rho$ is a Gauduchon metric, the adjoint of the Laplacian $\Delta_{\rho}$ has no zero$^{th}$-order terms. Thus, ${\cal Y}$ being compact, the maximum principle implies that \begin{eqnarray*}\Delta_{\rho}|\Gamma|^2_h\equiv 0 \hspace{3ex}\mbox{and}\hspace{3ex} |\Gamma|^2_h = Const \hspace{1ex} \mbox{on}\hspace{1ex} {\cal Y}.\end{eqnarray*}

  Alternatively, one can see directly that \begin{eqnarray*}\int\limits_{{\cal Y}}\bigg(\Delta_{\rho}|\Gamma|^2_h\bigg)\,dV_\rho =  \frac{1}{N-1}\,\int\limits_{{\cal Y}}i\bar\partial\partial\{\Gamma,\,\Gamma\}_h\wedge\rho^{N-1} = 0\end{eqnarray*} by the compactness of ${\cal Y}$, the Gauduchon property $i\bar\partial\partial\rho^{N-1} = 0$ and the Stokes theorem. As we have seen above, $\Delta_{\rho}|\Gamma|^2_h\leq 0$ everywhere on ${\cal Y}$, so this immediately yields $\Delta_{\rho}|\Gamma|^2_h = 0$ everywhere. 

  Moreover, under our hypothesis, we would have $\Delta_{\rho}|\Gamma|^2_h<0$ at $y_0\in{\cal Y}$ if $\Gamma(y_0)\neq 0$. We deduce that $\Gamma(y_0) = 0$, hence $\Gamma(y) = 0$ for every $y\in{\cal Y}$. This contradicts hypothesis (c) of the statement. \hfill $\Box$

\subsection{Curvature of ${\cal L}$ in the presence of a holomorphic $p$-contact structure}\label{subsection:L_curvature_p-contact}

We proved in [BP25] that a $p$-contact bundle $L$ cannot be semi-negative, even with respect to a possibly singular Hermitian fibre metric and even in a generalised, weaker, sense (called $(n-p)$-semi-negativity) on a compact K\"ahler manifold.  

\begin{The}(Theorem 3.5 of [BP25])\label{The:p-contact_n-p-pos} Let $L$ be a holomorphic line bundle on a compact {\bf K\"ahler} manifold $(X,\,\omega)$ with $\mbox{dim}_\C X = n = 2p+1$ and $p$ odd.

  If there exists an {\bf $L$-valued holomorphic $p$-contact structure} $\Gamma\in H^{p,\,0}_{\bar\partial}(X,\,L)$, there exists {\bf no} (singular or otherwise) Hermitian fibre metric $h^{-1}$ on the dual line bundle $L^{-1}$ (hence no such metric $h^{-2}$ on the canonical bundle $K_X\simeq L^{-2}$ of $X$) such that \begin{eqnarray*}i\Theta_{h^{-1}}(L^{-1})\geq_{n-p,\,\omega}0  \hspace{5ex} (\mbox{or equivalently} \hspace{3ex} i\Theta_{h^{-2}}(K_X)\geq_{n-p,\,\omega}0).\end{eqnarray*}

\end{The}  

An immediate consequence in our relative setting is the following

\begin{Cor}\label{Cor:L_curvature_p-contact} Let $\pi:{\cal X}\longrightarrow{\cal Y}$ be a proper holomorphic submersion between complex manifolds and let ${\cal L}$ be a holomorphic line bundle over ${\cal X}$. Assume that the fibres $X_y:=\pi^{-1}(y)$, with $y\in{\cal Y}$, are (compact) {\bf K\"ahler} manifolds of complex dimension $n=2p+1$, with $p$ odd.

  Assume there exists $y_0\in{\cal Y}$ such that the restriction $L_{y_0}:= {\cal L}_{|X_{y_0}}$ is a {\bf $p$-contact bundle}. Then, there exists {\bf no} (singular or otherwise) Hermitian fibre metric $h$ on ${\cal L}$, whose restriction $h_{y_0}$ to $L_{y_0}$ is non-degenerate, such that \begin{eqnarray*}i\Theta_h({\cal L})\leq_{n-p,\,\omega}0   \hspace{5ex} \mbox{everywhere on}\hspace{1ex} {\cal X}.\end{eqnarray*}

\end{Cor}  

\noindent {\it Proof.} If a fibre metric $h$ as in the statement existed on ${\cal L}$, the Chern curvature form of $L_{y_0}$ with respect to $h_{y_0}$ would have the property \begin{eqnarray*}i\Theta_{h_{y_0}}(L_{y_0})\leq_{n-p,\,\omega}0   \hspace{5ex} \mbox{everywhere on}\hspace{1ex}  X_{y_0}.\end{eqnarray*}

Since $L_{y_0}:= {\cal L}_{|X_{y_0}}$ is a $p$-contact bundle on the compact K\"ahler manifold $X_{y_0}$, this would contradict Theorem \ref{The:p-contact_n-p-pos}.  \hfill $\Box$

\subsection{(Semi-)positivity of $E:={\cal H}^{p,\,0}({\cal L})$}\label{subsection:E_positivity} We now return to the setting of Corollary \ref{Cor:p-contact-parallel_frame}. In particular, we assume the existence of an $L_t$-valued holomorphic $p$-contact structure $\Gamma_t\in H^{p,\,0}_{\bar\partial_t}(X_t,\,L_t)$ on $X_t$ for every $t\in B$. 

We consider the holomorphic vector bundle associated with the following direct-image sheaf: \begin{eqnarray}\label{eqn:Hn0-L2_def}F={\cal H}^{n,\,0}({\cal L}^2):=\pi_\star\bigg(\Omega^n_{{\cal X}/B}\otimes{\cal L}^2\bigg),\end{eqnarray} where we use the same notation $F$ for the bundle and the sheaf. Under general circumstances, we would have to assume that the non-negative integer $\mbox{dim}_\C H^{n,\,0}(X_t,\,L_t^2)$ is independent of $t\in B$ (i.e. the analogue of the non-jumping hypothesis $(H1)$) in order to guarantee that the above direct-image sheaf is locally free. However, our $p$-contact hypothesis implies rather more: \begin{eqnarray}\label{eqn:F-fibres_dim_1}\mbox{dim}_\C H^{n,\,0}(X_t,\,L_t^2) = 1, \hspace{6ex} t\in B.\end{eqnarray} Indeed, if $\Gamma_t\in H^{p,\,0}_{\bar\partial_t}(X_t,\,L_t)$ is an $L_t$-valued holomorphic $p$-contact structure on $X_t$, then $\Gamma_t\wedge\partial_t\Gamma_t\in H^{n,\,0}_{\bar\partial_t}(X_t,\,L_t^2)\simeq H^0(X_t,\,K_{X_t}\otimes L_t^2)\simeq\C$ (as argued in $\S$\ref{section:introd}, just above (\ref{eqn:gamma-map-homogeneous})).

Thus, $F={\cal H}^{n,\,0}({\cal L}^2)\longrightarrow B$ is the holomorphic vector bundle considered in [Ber09] (denoted there by $E$) and induced by ${\cal L}^2\longrightarrow{\cal X}$ here. In our $p$-contact situation, $F$ is a holomorphic {\it line bundle} whose fibres are \begin{eqnarray*}F_t= H^{n,\,0}(X_t,\,L_t^2), \hspace{5ex} t\in B.\end{eqnarray*} Note that, unlike [Ber09], we do not need to assume the existence of a K\"ahler metric on ${\cal X}$ at this point in order to ensure the existence of the line bundle $F$. Our $p$-contact hypothesis does that job.

As in [Ber09], we have the following analogue of Definition \ref{Def:h_E}. However, this definition requires no Hermitian metric on ${\cal X}$ -- see Remark \ref{Rem:L2_inner-product_p=n}.

\begin{Def}\label{Def:h_F} Let $h$ be a $C^\infty$ Hermitian fibre metric on ${\cal L}\longrightarrow{\cal X}$. We associate with $h$ the $C^\infty$ Hermitian fibre metric $h_F$ on $F={\cal H}^{n,\,0}({\cal L}^2)\longrightarrow B$ defined by $B\ni t\longmapsto\langle\langle\,\cdot\,,\, \,\cdot\,\rangle\rangle_{F_t,\,h_t^2}$, where $\langle\langle\,\cdot\,,\, \,\cdot\,\rangle\rangle_{F_t,\,h_t^2}$ is the $L^2_{h_t^2}$-inner product on $C^\infty_{n,\,0}(X_t,\,L_t^2)$ defined by \begin{eqnarray}\label{eqn:L2_inner-product_F}\langle\langle\Gamma_t,\, \Omega_t\rangle\rangle_{F_t,\,h_t^2}:= \int\limits_{X_t} i^{n^2}\,\{\Gamma_t,\, \Omega_t\}_{h_t^2}, \hspace{5ex} t\in B,\end{eqnarray} for all $\Gamma_t, \Omega_t\in C^\infty_{n,\,0}(X_t,\,L_t^2)$.

\end{Def}

Note that $i^{n^2}\,\{\Gamma_t,\, \Omega_t\}_{h_t^2}$ is a $\C$-valued $(n,\,n)$-form on $X_t$, so it can be integrated to produce a complex number. It is independent of any Hermitian metric on $X_t$. Formula (\ref{eqn:coupling_inner-prod_X}) translates to: \begin{eqnarray}\label{eqn:L2_inner-product_expression_F}i^{n^2}\,\{\Gamma_t,\, \Omega_t\}_{h_t^2} = \langle\Gamma_t,\, \Omega_t\rangle_{\omega_t,\,h_t^2}\,dV_{\omega_t}\end{eqnarray} for every Hermitian metric $\omega_t$ on $X_t$. In particular, the right-hand side of (\ref{eqn:L2_inner-product_expression_F}) is independent of $\omega_t$.

\vspace{2ex}

Let $D_{\cal L} = D_{\cal L}' + D_{\cal L}''$ be the Chern connection of $({\cal L},\,h)$. On the other hand, we denote by $D_L = D_L' + D_L'' = (D_{L_t}')_{t\in B} + (D_{L_t}'')_{t\in B}$ the fibrewise Chern connection of $(L_t)_{t\in B}$. This means that, for each $t\in B$, $D_{L_t} = D_{L_t}' + D_{L_t}''$ is the Chern connection of $(L_t,\,h_t)$.

In particular, it makes sense to apply $D_L$ to any local $C^\infty$ section $u$ of $E$: if $u\in C^\infty(U,\,E)$, for some open subset $U\subset B$, then $u(t)\in E_t=H^{p,\,0}_{\bar\partial_t}(X_t,\,L_t)$, so $u(t)$ is a holomorphic $L_t$-valued $(p,\,0)$-form on $X_t$ and, therefore, \begin{eqnarray*}D_{L_t}u(t) = D_{L_t}'u(t) + D_{L_t}''u(t) = D_{L_t}'u(t)\in C^\infty_{p+1,\,0}(X_t,\,L_t), \hspace{5ex} t\in U.\end{eqnarray*}

Thus, we get: $D_Lu = D'_Lu = \bigg(D_{L_t}'u(t)\bigg)_{t\in U}$. Moreover, for any section $v\in C^\infty(U,\,E)$, we get: \begin{eqnarray*}\frac{1}{2}\,\bigg(u\wedge D'_Lv + D'_Lu\wedge v\bigg) = \frac{1}{2}\,\bigg(u(t)\wedge D_{L_t}'v(t) + D_{L_t}'u(t)\wedge v(t)\bigg)_{t\in U}\in C^\infty(U,\,F).\end{eqnarray*} Indeed, the right-hand-side term above is a {\it holomorphic} $L_t^2$-valued $(n,\,0)$-form on $X_t$ for every $t\in U$. To see the holomorphicity, we use the relations $D''_{L_t}u(t) = D''_{L_t}v(t) =0$ to compute: \begin{eqnarray*}D''_{L_t^2}\bigg(u(t)\wedge D_{L_t}'v(t) + D_{L_t}'u(t)\wedge v(t)\bigg) & = & -u(t)\wedge D_{L_t}''D_{L_t}'v(t) +  D_{L_t}''D_{L_t}'u(t)\wedge v(t) \\
  & = & -u(t)\wedge\Theta_{h_t}(L_t)\wedge v(t) + \Theta_{h_t}(L_t)\wedge u(t)\wedge v(t) \\
  & = & \Theta_{h_t}(L_t)\wedge\bigg(-u(t)\wedge v(t) + u(t)\wedge v(t)\bigg) = 0,\end{eqnarray*} where we have also used the oddness of $p$.

When $u=v\in C^\infty(U,\,E)$, the evenness of $p+1$ yields $u\wedge D'_Lu = D'_Lu\wedge u$. Thus, $u(t)\wedge D_{L_t}'u(t)$ is holomorphic (as an $L_t^2$-valued $(n,\,0)$-form on $X_t$), making $u\wedge D'_Lu$ a $C^\infty$ section of $F$.

\vspace{1ex}

These observations enable us to make the following

\begin{Def}\label{Def:bilinear-form} (i)\, We consider the following {\bf holomorphic quadratic} bundle map: \begin{eqnarray*}S:E\longrightarrow F, \hspace{5ex} S(u):=u\wedge D'_Lu.\end{eqnarray*}

\vspace{1ex}

  (ii)\, We consider the following {\bf holomorphic symmetric bilinear} bundle map: \begin{eqnarray*}A:E\times E\longrightarrow F, \hspace{5ex} A(u,\,v):=\frac{1}{2}\,\bigg(u\wedge D'_Lv + D'_Lu\wedge v\bigg).\end{eqnarray*}

\end{Def}

For all sections $u, v\in C^\infty(U,\,E)$, we have: \begin{eqnarray*}A(u,\,v) = \frac{1}{2}\,\bigg(S(u+v) - S(u) - S(v)\bigg), \hspace{3ex} \mbox{in particular}\hspace{3ex} A(u,\,u) = S(u),\end{eqnarray*} so $A$ is the {\it polarisation} of $S$. Note that if $\Gamma\in C^\infty(U,\,E)$ is a {\bf $p$-contact section} (in the sense that $\Gamma(t)\in H^{p,\,0}_{\bar\partial_t}(X_t,\,L_t)$ is $p$-contact for all $t\in U$), the section $A(\Gamma,\,\Gamma)\in C^\infty(U,\,F)$ is {\it non-vanishing}: \begin{eqnarray}\label{eqn:p-contact_section_A-pos}A(\Gamma,\,\Gamma)(t) = S(\Gamma)(t) = \Gamma(t)\wedge D'_{L_t}\Gamma(t)\neq 0,  \hspace{5ex} t\in U.\end{eqnarray}

\begin{Cor}\label{Cor:1-1_form_E-curvature} For any open subset $U\subset B$ and any {\bf $p$-contact section} $\Gamma\in C^\infty(U,\,E)$, \begin{eqnarray*}\frac{1}{A(\Gamma,\,\Gamma)}\,A\bigg(i\Theta_{h_E}(E)\,\Gamma,\,\Gamma\bigg)\end{eqnarray*} is a well-defined $\C$-valued $(1,\,1)$-form on $U$.

\end{Cor}

A word of explanation is in order. We have seen that $A(\Gamma,\,\Gamma)$ is not a function on $U$, but a smooth section of $F$ on $U$. However, once a local holomorphic trivialisation (or frame) of $F$ has been fixed, $A(\Gamma,\,\Gamma)$ becomes a function with respect to it. This local function being non-vanishing (cf. (\ref{eqn:p-contact_section_A-pos})), we can divide by it. On the other hand, $i\Theta_{h_E}(E)$ is an $\operatorname{End} E$-valued $(1,\,1)$-form on $B$, so $i\Theta_{h_E}(E)\,\Gamma$ is an $E$-valued $(1,\,1)$-form on $U$. It follows that $A(i\Theta_{h_E}(E)\,\Gamma,\,\Gamma)$ is an $F$-valued $(1,\,1)$-form on $U$. In the chosen local holomorphic trivialisation of $F$, it identifies with a $\C$-valued $(1,\,1)$-form on $U$. The $\C$-valued $(1,\,1)$-form on $U$ considered in Corollary \ref{Cor:1-1_form_E-curvature} is independent of the chosen local holomorphic trivialisation of $F$. It is in this sense that it is {\it well defined}.

\begin{Def}\label{Def:A-twisted_Griffiths} The Hermitian holomorphic vector bundle $(E,\,h_E)\longrightarrow B$ is said to be {\bf A-twisted Griffiths (semi-)positive} if for every open subset $U\subset B$ and every $p$-contact section $\Gamma\in C^\infty(U,\,E)$, we have: \begin{eqnarray*}\frac{1}{A(\Gamma,\,\Gamma)}\,A\bigg(i\Theta_{h_E}(E)\,\Gamma,\,\Gamma\bigg) > 0 \hspace{3ex} (\mbox{respectively \hspace{1ex} $\geq 0$})\end{eqnarray*} as a $\C$-valued $(1,\,1)$-form on $U$. In this case, we write:

  \vspace{1ex}

\hspace{6ex} $(E,\,h_E)> 0$ (A-twisted Griffiths), respectively $(E,\,h_E)\geq 0$ (A-twisted Griffiths).

\end{Def}

The above definition can be formulated in a pointwise manner: for every $t\in B$, it requires {\it positive definiteness} (respectively {\it semi-positivity}) at $t$ of the $\C$-valued $(1,\,1)$-form of Corollary \ref{Cor:1-1_form_E-curvature} for every holomorphic $p$-contact form $\Gamma_t\in H^{p,\,0}(X_t,\,L_t)$.

\begin{The}\label{The:F-E_positivity} Let $\pi:{\cal X}\longrightarrow B$ be a holomorphic family of compact complex manifolds of dimension $n=2p+1=4l+3$ over an arbitrary complex manifold base $B$. Let ${\cal L}\longrightarrow{\cal X}$ be a holomorphic line bundle satisfying assumption $(H1)$ globally on $B$. Suppose there exists an $L_t$-valued holomorphic $p$-contact structure $\Gamma_t\in H^{p,\,0}_{\bar\partial_t}(X_t,\,L_t)$ on $X_t$ for every $t\in B$.

  Then, for any $C^\infty$ Hermitian fibre metric $h$ on ${\cal L}\longrightarrow{\cal X}$ and any Hermitian metric $\omega$ on ${\cal X}$, denoting by $h_E$ and $h_F$ the induced $C^\infty$ Hermitian fibre metrics on $E$, respectively $F$, the following implications hold: \begin{eqnarray*}i\Theta_{h_F}(F)\geq 0 & \implies & (E,\,h_E)\geq 0 \hspace{2ex} (\mbox{\bf A-twisted Griffiths})\\
 i\Theta_{h_F}(F)> 0 & \implies & (E,\,h_E)> 0 \hspace{2ex} (\mbox{\bf A-twisted Griffiths}).\end{eqnarray*}

\end{The}

\noindent {\it Proof.} For any open subset $U\subset B$ and any sections $u,v\in C^\infty(U,\,E)$, we get: \begin{eqnarray*}D'_F A(u,\,v) & = & \frac{1}{2}\,\bigg(D'_E u\wedge D'_Lv + u\wedge D'_E(D'_Lv) + D'_E(D'_Lu)\wedge v + D'_Lu\wedge D'_Ev\bigg) \\
  & = & \frac{1}{2}\,\bigg(D'_E u\wedge D'_Lv + u\wedge D'_L(D'_Ev) + D'_L(D'_E u)\wedge v + D'_Lu\wedge D'_E v\bigg) \\
  & = & A(D'_E u,\,v) + A(u,\,D'_Ev).\end{eqnarray*} We have used the commutation of $D'_E$ (which acts {\it horizontally}, i.e. in the $B$-directions, in the $t$-variables) with $D'_L$ (which acts {\it vertically}, i.e. in the directions of the fibres $X_t$). The analogous identity for $D''_F=\bar\partial$ in place of $D'_F$ can be proved in the same way.

Now, suppose the sections $u$ and $v$ are {\it holomorphic}. Thus, $D''_E u = D''_E v = 0$ on $U$. Applying $D''_F$ to the above identity, we get the second equality below: \begin{eqnarray}\label{eqn:curvatura_F_u-v_A}\Theta_{h_F}(F)\bigg(A(u,\,v)\bigg) = D''_FD'_F A(u,\,v) = A\bigg(\Theta_{h_E}(E)\, u,\,v\bigg) + A\bigg(u,\,\Theta_{h_E}(E)\,v\bigg).\end{eqnarray}

If, moreover, the holomorphic sections $u$ and $v$ are a same {\bf $p$-contact section} $\Gamma\in H^0(U,\,E)$, then $A(\Gamma,\,\Gamma)\in H^0(U,\,F)$ is a {\it non-vanishing} holomorphic section of $F$ (see (\ref{eqn:p-contact_section_A-pos})) and we get: \begin{eqnarray}\label{eqn:curvatura_F_Gamma_A}i\Theta_{h_F}(F) = \frac{2}{A(\Gamma,\,\Gamma)}\,A\bigg(i\Theta_{h_E}(E)\,\Gamma,\,\Gamma\bigg)\end{eqnarray} after choosing $A(\Gamma,\,\Gamma)$ as a local holomorphic frame of $F$. Note that the right-hand-side term of (\ref{eqn:curvatura_F_Gamma_A}) is twice the object of Corollary \ref{Cor:1-1_form_E-curvature}.

The contention follows from (\ref{eqn:curvatura_F_Gamma_A}) and from Definition \ref{Def:A-twisted_Griffiths}. The fact that $\Gamma$ was assumed {\it holomorphic} in the proof of (\ref{eqn:curvatura_F_Gamma_A}) and only $C^\infty$ in Definition \ref{Def:A-twisted_Griffiths} is not problematic since (\ref{eqn:curvatura_F_Gamma_A}) and the notion of {\it A-twisted Griffiths (semi-)positivity} at a point $t\in B$ depend only on the value $\Gamma(t)$.  \hfill $\Box$

\vspace{2ex}

Theorem \ref{The:F-E_positivity} can be combined with Theorem 1.2 of [Ber09]. We now state the latter in the case of the holomorphic line bundle $F\longrightarrow B$ (whose rank one is due to our {\it holomorphic $p$-contact} hypothesis), but Berndtsson's theorem holds for vector bundles of arbitrary rank that are proved to be {\bf Nakano (semi-)positive} under the (semi-)positivity assumption on ${\cal L}$.

\begin{The}([Ber09, Theorem 1.2])\label{The:Berndtsson} If the complex manifold ${\cal X}$ is {\bf K\"ahler} and ${\cal L}\rightarrow{\cal X}$ has a $C^\infty$ Hermitian fibre metric $h$ whose Chern curvature form $i\Theta_h({\cal L})$ is {\bf (semi-)positive} everywhere on ${\cal X}$, then the Hermitian holomorphic line bundle $(F,\, h_F)$ is {\bf (semi-)positive} everywhere on $B$.

\end{The}

For line bundles (such as $F$), both Nakano and Griffiths (semi-)positivity properties coincide with the standard (semi-)positivity property of the Chern curvature form of the line bundle. On the other hand, $F$ is defined here using ${\cal L}^2$ rather than ${\cal L}$, but $i\Theta_h({\cal L}^2) = 2\,i\Theta_h({\cal L})$, so the positive factor is harmless for positivity.

Putting together Theorems \ref{The:F-E_positivity} and \ref{The:Berndtsson}, we get Theorem \ref{The:introd_F_positivity_Berndtsson-application} stated in the introduction.

\section{Moduli space of relative holomorphic $p$-contact structures}\label{section:moduli}

Let $\pi:{\cal X}\longrightarrow B$ be a {\it proper holomorphic submersion}, where ${\cal X}$ and $B$ are complex manifolds. Suppose the complex dimension of the fibres $X_t:=\pi^{-1}(t)$, with $t\in B$, is $n=2p+1$ for some {\it odd} positive integer $p$. Let ${\cal L}\longrightarrow{\cal X}$ be a $C^\infty$ complex line bundle and let $L_t:= {\cal L}_{|X_t}$ be its restriction to the fibre over $t$.

We denote by ${\cal H}ol({\cal L})$ the topological space consisting of the set of holomorphic structures $D''$ on ${\cal L}$ endowed with the $C^\infty$ topology. For any $D''\in{\cal H}ol({\cal L})$, let ${\cal L}_{D''}:=({\cal L},\,D'')$ be the resulting holomorphic line bundle over ${\cal X}$ and let $L_{D'',\,t}:={\cal L}_{D''|X_t}\longrightarrow X_t$ be the restriction of ${\cal L}_{D''}$ to the fibre over $t$. We denote by $D''_t:=(D'')_{|X_t}$ the holomorphic structure of $L_{D'',\,t}$ (induced by $D''$). Thus, $L_{D'',\,t} = (L_t,\,D''_t)$.

If $D''$ satisfies hypothesis $(H1)$ globally on $B$, namely \begin{eqnarray*} (H1)\hspace{3ex} \mbox{dim}_\C H^{p,\,0}_{D''_t}(X_t,\,L_t) \hspace{3ex}\mbox{is independent of}\hspace{1ex} t\in B,\end{eqnarray*} let $E_{D''}:={\cal H}^{p,\,0}_{D''}({\cal L})$ be the holomorphic vector bundle over $B$ associated with the locally free sheaf \begin{eqnarray*}\Omega^p_{{\cal X}/B}\otimes{\cal L}_{D''}.\end{eqnarray*} As we saw in $\S$\ref{subsection:Hp0_line-bundle}, its fibres are \begin{eqnarray*}(E_{D''})_t = {\cal H}^{p,\,0}_{D''}({\cal L})_t = H^{p,\,0}_{D''_t}(X_t,\,L_t),  \hspace{5ex} t\in B.\end{eqnarray*}

A continuous section $\Gamma\in C^0(U,\,{\cal H}^{p,\,0}_{D''}({\cal L})_{|U})$ on some open subset $U\subset B$ is said to be a {\bf continuous $p$-contact section} if, for every $t\in U$, the $L_{D'',\, t}$-valued holomorphic $(p,\,0)$-form $\Gamma_t:=\Gamma(t)$ on $X_t$ satisfies the condition: \begin{eqnarray*}\Gamma_t\wedge\partial_t\Gamma_t\neq 0 \hspace{5ex} \mbox{at every point of $X_t$}.\end{eqnarray*}

For every fixed $r\in\N$, we will say that a holomorphic structure $D''\in{\cal H}ol({\cal L})$ satisfies {\bf hypothesis $(H_1^{(r)})$} globally on $B$ if \begin{eqnarray*} (H_1^{(r)})\hspace{5ex} \mbox{dim}_\C H^{p,\,0}_{D''_t}(X_t,\,L_t)=r \hspace{3ex}\mbox{for all}\hspace{1ex} t\in B.\end{eqnarray*} 

\begin{Def}\label{Def:C_p_set} For each $r\in\N$, we consider the topological space consisting of the set \begin{eqnarray*}{\cal C}^{(r)}_p({\cal X},\,B,\,{\cal L}):=\bigg\{D''\in{\cal H}ol({\cal L}) &\mid & D''\hspace{1ex}\mbox{satisfies}\hspace{1ex} (H_1^{(r)}) \hspace{1ex} \mbox{globally on $B$ and, for every $t_0\in B$, there} \\
 & & \mbox{exists an open neighbourhood $U$ of $t_0$, with compact closure in $B$,} \\
  & & \mbox{and a continuous $p$-contact section}\hspace{1ex} \Gamma\in C^0\bigg(U,\,{\cal H}^{p,\,0}_{D''}({\cal L})_{|U}\bigg)\bigg\}\end{eqnarray*} equipped with the $C^\infty$ topology.

\end{Def}

For any $C^\infty$ Hermitian fibre metric $h$ on ${\cal L}$, we denote by $U_h({\cal L})$ the group of $h$-unitary $C^\infty$ automorphisms of ${\cal L}$. In particular, each $T\in U_h({\cal L})$ defines an isometry $T_{(t,\,z)}$ of each fibre ${\cal L}_{(t,\,z)}$ for the pointwise inner product $\langle\,\cdot\,,\,\cdot\,\rangle_{h(t,\,z)}$. For each $t\in B$, each $T\in U_h({\cal L})$ induces an $h_t$-unitary $C^\infty$ automorphism $T_t\in U_{h_t}(L_t)$ of $L_t={\cal L}_{|X_t}$, where $h_t:=h_{|L_t}$.

\begin{Lem}\label{Lem:right-action_C_p} For each $r\in\N$ and any $C^\infty$ Hermitian fibre metric $h$ on ${\cal L}$, the right action: \begin{eqnarray*}{\cal C}^{(r)}_p({\cal X},\,B,\,{\cal L})\times U_h({\cal L})\longrightarrow {\cal C}^{(r)}_p({\cal X},\,B,\,{\cal L}), \hspace{5ex} (D'',\,T)\longmapsto D''_T:=T^{-1}\circ D''\circ T,\end{eqnarray*} is {\bf well defined} in the following sense:

\vspace{1ex}

(a)\, if $D''$ satisfies the non-jumping property $(H1)$ globally on $B$, then so does $D''_T$;

\vspace{1ex}

(b)\, for any open subset $U\subset B$, if $\Gamma\in C^0\bigg(U,\,{\cal H}^{p,\,0}_{D''}({\cal L})\bigg)$ is a continuous $p$-contact section, then $T^{-1}\Gamma\in C^0\bigg(U,\,{\cal H}^{p,\,0}_{D''_T}({\cal L})\bigg)$ is a continuous $p$-contact section for every $T\in U_h({\cal L})$.

\end{Lem}

\noindent {\it Proof.} (a)\, For each $t\in B$, we have $D''_{T,\,t} = T_t^{-1}\circ D''_t\circ T_t$. This readily implies the following identities: \begin{eqnarray*}T_t^{-1}(\ker D''_t) = \ker D''_{T,\,t} \hspace{5ex}\mbox{and}\hspace{5ex} T_t^{-1}(\mbox{Im}\, D''_t) = \mbox{Im}\, D''_{T,\,t}.\end{eqnarray*} We infer that \begin{eqnarray*}T_t^{-1}:H^{p,\,0}_{D''_t}(X_t,\,L_t)\longrightarrow H^{p,\,0}_{D''_{T,\,t}}(X_t,\,L_t)\end{eqnarray*} is a $\C$-vector space isomorphism for every $t\in B$. Thus, these vector spaces have the same dimension.

\vspace{1ex}

(b)\, In local holomorphic coordinates $z_1^{(t)},\dots , z_n^{(t)}$ on $X_t$ and a local holomorphic frame $\{e_t\}$ of $L_{D'',\,t} = (L_t,\,D''_t)$, we have: \begin{eqnarray*}\Gamma_t = \sum\limits_{|I|=p}\Gamma_I^{(t)}\,dz_I^{(t)}\otimes e_t, \hspace{3ex}\mbox{hence}\hspace{3ex} T_t^{-1}\Gamma_t = \sum\limits_{|I|=p}\Gamma_I^{(t)}\,dz_I^{(t)}\otimes (T_t^{-1}e_t),\end{eqnarray*} where the $\Gamma_I^{(t)}$ are locally defined holomorphic functions on $X_t$. Consequently, we get the local expressions: \begin{eqnarray*}\Gamma_t\wedge\partial_t\Gamma_t & = & \bigg(\sum\limits_{|I|=|J|=p}\Gamma_I^{(t)}\,dz_I^{(t)}\wedge\partial_t\Gamma_J^{(t)}\wedge dz_J^{(t)} \bigg)\otimes e_t^2, \\
  \bigg(T_t^{-1}\Gamma_t\bigg)\wedge\partial_t\bigg(T_t^{-1}\Gamma_t\bigg) & = & T_t^{-2}\bigg(\Gamma_t\wedge\partial_t\Gamma_t\bigg),\end{eqnarray*} for every $t\in B$. This proves (b).  \hfill $\Box$

\vspace{2ex}

Lemma \ref{Lem:right-action_C_p} enables us to make the following

\begin{Def}\label{Def:moduli-space}  For each $r\in\N$, the quotient set ${\cal M}^{(r)}_p({\cal X},\,B,\,{\cal L}) :={\cal C}^{(r)}_p({\cal X},\,B,\,{\cal L})/U_h({\cal L})$, equipped with the quotient topology, is called the {\bf r-th moduli space of relative holomorphic p-contact structures} on $({\cal X},\,B,\,{\cal L})$.

\end{Def}

Two of its properties are given in

\begin{The}\label{The:moduli-space_prop} Fix an arbitrary $r\in\N$.

\vspace{1ex}  

  $(1)$ The moduli space ${\cal M}^{(r)}_p({\cal X},\,B,\,{\cal L})$ is {\bf Hausdorff}.

\vspace{1ex}

$(2)$\, If $B$ is {\bf compact}, the moduli space ${\cal M}^{(r)}_p({\cal X},\,B,\,{\cal L})$ is {\bf open} in the quotient topological space ${\cal P}^{(r)}_p({\cal X},\,B,\,{\cal L})/U_h({\cal L})$, where ${\cal P}^{(r)}_p({\cal X},\,B,\,{\cal L})$ is the topological space consisting of the set \begin{eqnarray*}{\cal P}^{(r)}_p({\cal X},\,B,\,{\cal L}):= \bigg\{D''\in{\cal H}ol({\cal L}) &\mid & D''\hspace{1ex}\mbox{satisfies}\hspace{1ex} (H_1^{(r)})\hspace{1ex} \mbox{globally on $B$}\bigg\}\end{eqnarray*} equipped with the $C^\infty$ topology.

\end{The}

\noindent {\it Proof.} $(1)$\, It is standard (see e.g. [Kob87, Corollary 7.1.15]) that the classical moduli space ${\cal H}ol_h({\cal L})/U_h({\cal L})$ is {\it Hausdorff}, where ${\cal H}ol_h({\cal L})$ denotes the topological space of $h$-compatible Chern connections on ${\cal L}$. Equivalently, ${\cal H}ol_h({\cal L})$ is the topological space (with $C^\infty$ topology) of pairs $(D'',\,h)$ of a holomorphic structure on ${\cal L}$ and the fixed Hermitian fibre metric $h$. Thus, ${\cal H}ol_h({\cal L})$ identifies with ${\cal H}ol({\cal L})$, so ${\cal C}^{(r)}_p({\cal X},\,B,\,{\cal L})\subset{\cal H}ol_h({\cal L})$ is a topological subspace.

It follows that ${\cal M}^{(r)}_p({\cal X},\,B,\,{\cal L})\subset{\cal H}ol_h({\cal L})/U_h({\cal L})$ is a topological subspace (as we equip both quotients with the quotient topologies). Since every topological subspace of a Hausdorff topological space is Hausdorff, the contention follows.

\vspace{1ex}

$(2)$\, Let $D''\in{\cal C}^{(r)}_p({\cal X},\,B,\,{\cal L})$. Fix an arbitrary point $t_0\in B$. Then, there exist an open neighbourhood $U\subset B$ of $t_0$ such that $\overline{U}$ is compact in $B$ and a continuous section $\Gamma\in C^0\bigg(U,\,{\cal H}^{p,\,0}_{D''}({\cal L})_{|U}\bigg)$ such that \begin{eqnarray*}\Gamma_t\wedge\partial_t\Gamma_t\neq 0 \hspace{5ex} \mbox{at every point of $X_t$}, \hspace{5ex} t\in U.\end{eqnarray*}

This implies that $\widetilde\Gamma\in C^0({\cal X}_U,\,K_{{\cal X}/B}\otimes{\cal L}^2)$ is a {\it non-vanishing} continuous section on ${\cal X}_U:=\pi^{-1}(U)\subset{\cal X}$ of the holomorphic line bundle $(K_{{\cal X}/B}\otimes{\cal L}^2,\,D'')\longrightarrow{\cal X}$ whose holomorphic structure (still denoted by $D''$) is induced by the given $D''\in{\cal C}^{(r)}_p({\cal X},\,B,\,{\cal L})$, where $\widetilde\Gamma$ is defined by: \begin{eqnarray*}\widetilde\Gamma(t,\,z):= \Gamma_t(z)\wedge(\partial_t\Gamma_t)(z)\in\bigg(K_{{\cal X}/B}\otimes{\cal L}^2\bigg)_{(t,\,z)} = (K_{X_t}\otimes L_t^2)_z, \hspace{5ex} (t,\,z)\in{\cal X}\hspace{1ex} (\mbox{i.e.}\hspace{1ex} t\in B,\hspace{1ex} z\in X_t).\end{eqnarray*}

Now, suppose that we have a continuous family of holomorphic structures $D''_s\in{\cal P}^{(r)}_p({\cal X},\,B,\,{\cal L})$ close enough to $D''\in{\cal C}^{(r)}_p({\cal X},\,B,\,{\cal L})$, with $s$ varying in an open subset of $V\subset{\cal P}^{(r)}_p({\cal X},\,B,\,{\cal L})$ that is a (small enough) neighbourhood of the point $s_0\in{\cal P}^{(r)}_p({\cal X},\,B,\,{\cal L})$ corresponding to $D''$. (Thus, $D''=D''_{s_0}$.) The non-jumping assumption $(H_1^{(r)})$ in the definition of ${\cal P}^{(r)}_p({\cal X},\,B,\,{\cal L})$ yields a continuous family ${\cal H}^{p,\,0}_{D''_s}({\cal L})\longrightarrow B$ of holomorphic vector bundles whose member for $s=s_0$ is ${\cal H}^{p,\,0}_{D''}({\cal L})$. This continuous family of vector bundles identifies naturally with a continuous vector bundle ${\cal H}^{p,\,0}({\cal L})$ over $B\times V$ such that, for every $s\in V$, we have ${\cal H}^{p,\,0}({\cal L})_{|B\times\{s\}} = {\cal H}^{p,\,0}_{D''_s}({\cal L})$.

Together with the compactness of $\overline{U}$ in $B$, this implies that the original section \begin{eqnarray*}\Gamma\in C^0\bigg(U,\,{\cal H}^{p,\,0}_{D''}({\cal L})_{|U}\bigg) = C^0\bigg(U,\,{\cal H}^{p,\,0}({\cal L})_{|U\times\{s_0\}}\bigg)\end{eqnarray*} extends continuously to sections $\Gamma^{(s)}\in C^0\bigg(U,\,{\cal H}^{p,\,0}_{D''_s}({\cal L})_{|U}\bigg) = C^0\bigg(U,\,{\cal H}^{p,\,0}({\cal L})_{|U\times\{s\}}\bigg)$ for all $D''_s$ close enough to $D'' = D''_{s_0}$ (equivalently, for all $s$ close enough to $s_0$) in the $C^\infty$ topology of ${\cal P}^{(r)}_p({\cal X},\,B,\,{\cal L})$. In other words, we get a section $\Gamma^{(s)}$ for every $s\in V\subset{\cal P}^{(r)}_p({\cal X},\,B,\,{\cal L})$ after possibly shrinking $V$ about $s_0$. 

Similarly, $(K_{{\cal X}/B}\otimes{\cal L}^2,\,D''_s)\longrightarrow{\cal X}$ is a continuous family (with $s\in V$) of holomorphic line bundles (for the induced holomorphic structures, still denoted by $D''_s$) whose member for $s=s_0$ is the original $(K_{{\cal X}/B}\otimes{\cal L}^2,\,D'')\longrightarrow{\cal X}$. Since $\overline{U}$ is compact, so is ${\cal X}_{\overline{U}}$. Therefore, the {\it non-vanishing} of the section $\widetilde\Gamma\in C^0({\cal X}_U,\,(K_{{\cal X}/B}\otimes{\cal L}^2,\,D''))$ propagates to the sections $\widetilde{\Gamma^{(s)}}\in C^0({\cal X}_U,\,(K_{{\cal X}/B}\otimes{\cal L}^2,\,D''_s))$, induced by the $\Gamma^{(s)}$ and thus depending continuously on $D''_s$, for all $D''_s$ close enough to $D''$ in the $C^\infty$ topology of ${\cal P}^{(r)}_p({\cal X},\,B,\,{\cal L})$. (So, we may need to further shrink $V$ about $s_0$.)

The {\it non-vanishing} of $\widetilde{\Gamma^{(s)}}\in C^0({\cal X}_U,\,(K_{{\cal X}/B}\otimes{\cal L}^2,\,D''_s))$, for all $D''_s$ close enough to $D''$ in the $C^\infty$ topology of ${\cal P}^{(r)}_p({\cal X},\,B,\,{\cal L})$, means that $\Gamma^{(s)}\in C^0\bigg(U,\,{\cal H}^{p,\,0}_{D''_s}({\cal L})_{|U}\bigg)$ is a continuous $p$-contact section of ${\cal H}^{p,\,0}_{D''_s}({\cal L})$ for all $s\in V$.

Note, however, that this $V\subset{\cal P}^{(r)}_p({\cal X},\,B,\,{\cal L})$ depends on $U\subset B$, so $V = V_U$. Since $B$ is compact, we can cover it by finitely many such open subsets $U_1,\dots , U_N$. For each $U_j$, we get an open neighbourhood $V_j\subset{\cal P}^{(r)}_p({\cal X},\,B,\,{\cal L})$ of $s_0$ and take $V:=V_1\cap\dots V_N$. This proves that $D''_s\in{\cal C}^{(r)}_p({\cal X},\,B,\,{\cal L})$ for all $s\in V$. \hfill $\Box$

\vspace{6ex}

\noindent {\bf References.} \\

\noindent [Ber09]\, B. Berndtsson --- {\it Curvature of Vector Bundles Associated to Holomorphic Fibrations} --- Ann. Math. {\bf 169} (2009), 531-560. 

\vspace{1ex}

\noindent [BP25]\, K. Broder, D. Popovici --- {\it Holomorphic $p$-Contact and $s$-Symplectic Line Bundles} --- J. Geom. Anal. (2026), 36:137, https://doi.org/10.1007/s12220-026-02380-6 (special issue dedicated to Josip Globevnik for his 80th birthday).

\vspace{1ex}

\noindent [KPU25a]\, H. Kasuya, D. Popovici, L. Ugarte --- {\it Higher-Degree Holomorphic Contact Structures} ---  arXiv:2502.01447v2 [math.DG].

\vspace{1ex}

\noindent [KPU25b]\, H. Kasuya, D. Popovici, L. Ugarte --- {\it Properties of Holomorphic p-Contact Manifolds} --- arXiv:2511.10818v1 [math.DG].

\vspace{1ex}

\noindent [Kob87]\, S. Kobayashi --- {\it Differential Geometry of Complex Vector Bundles} --- Princeton University Press, 1987.

\vspace{1ex}

\noindent [KS60]\, K. Kodaira, D.C. Spencer --- {\it On Deformations of Complex Analytic Structures, III. Stability Theorems for Complex Structures} --- Ann. of Math. {\bf 71}, no.1 (1960), 43-76.

\vspace{1ex}

\noindent [Pop26]\, D. Popovici --- {\it $m$-Positive Stability of Holomorphic Vector Bundles and Moduli Spaces} --- arXiv:2607.17203v1 [math.DG]

\vspace{6ex}

\noindent Institut de Math\'ematiques de Toulouse, Universit\'e de Toulouse,

\noindent 118 route de Narbonne, 31062 Toulouse, France

\noindent Email: popovici@math.univ-toulouse.fr

\end{document}